\documentclass[a4paper,11pt]{amsart}
\usepackage[T1]{fontenc}

\usepackage{amssymb}
\usepackage{amsmath}
\usepackage{enumitem}

\newtheorem{theorem}{Theorem}
\newtheorem{corollary}[theorem]{Corollary}
\newtheorem{lemma}[theorem]{Lemma}
\newtheorem{proposition}[theorem]{Proposition}

\theoremstyle{remark}
\newtheorem{remark}[theorem]{Remark}

\newcommand{\F}{{\mathbb{F}_{q}}}
\newcommand{\K}{\mathbb{K}}

\begin{document}

\title[Reducing the number of equations over a finite field]{Reducing the number of equations defining a subset of the $n$-space over a finite field}

\author{Stefan Bara\'{n}czuk} 
\address{Collegium Mathematicum, Adam Mickiewicz University, ul. Uniwersytetu Pozna\'{n}skiego 4,
	61-614, Pozna\'{n}, Poland}
\email{stefbar@amu.edu.pl}
\begin{abstract} Let $f_{1}, \ldots, f_{k}$ be polynomials defining an algebraic set in affine  $n$-space over a finite field. Suppose $k>n$. We  prove that there exists a system of polynomials $g_{1}, \ldots, g_{n}$, each being a linear combination with scalar coefficients of $f_{1}, \ldots, f_{k}$, defining the same algebraic set. In particular, one reduces the number of equations without increasing the total degree.  We also have the corresponding result for systems of homogeneous polynomials defining algebraic sets in projective spaces. 
\end{abstract}
\keywords{finite fields; algebraic sets; defining polynomials; reduction} 
\subjclass[2020]{11G25, 14A25}

\maketitle

The theorem that any algebraic set in $n$-dimensional space is  the intersection of $n$ hypersurfaces  \footnote{The problem dates back to Kronecker. Its rather dramatic story is briefly presented in \cite{EE}; for much more detailed vivid account consult N. Schappacher's available online presentation \textit{Political Space Curves}.}  has been proved independently by  Storch (\cite{S}), and Eisenbud and Evans (\cite{EE}); both short proofs are ring-theoretic, i.e., one reduces the number of generators of radical ideals.

In this note we examine closer the finite fields case of the problem. If just the number of equations needed to describe an algebraic set is in question, then the answer is immediate: it is easy to construct a single polynomial defining it. If, however, the nature of defining polynomials (e.g., their total degree) is to be preserved, this problem becomes more interesting.

 It turns out that we can avoid dealing with rings; the vector space structure is sufficient and, as in the
 theorem cited above, our result again produces $n$ equations; moreover, we show
 that these new equations can be chosen to be linear combinations with \textit{scalar} coefficients of the old ones, so, roughly speaking, they remain of the same type (see Corollaries \ref{wniosek} and \ref{wniosek1}, with accompanying examples), and our proof is surprisingly elementary.   \\

We fix the following notation:\\
\begin{tabular}{ll}
	$\F$ &   the finite field with $q$ elements; \\
	$\mathrm{Map}(X,\F)$ & the vector space of all functions $f : X \to \F$ for a given set $X$; \\
$\mathrm{Z}(f_{1}, \ldots, f_{k})$ & the set of common zeros of $f_{1}, \ldots, f_{k} \in \mathrm{Map}(X,\F)$; \\
$\mathrm{Span}(f_{1}, \ldots, f_{k})$ & the subspace of $\mathrm{Map}(X,\F)$ generated by $f_{1}, \ldots, f_{k}$;\\
$\mathbb{A}^{n}(\K)$ & the affine $n$-space over a field $\K$;\\
$\mathbb{P}^{n}(\K)$ & the projective $n$-space over a field $\K$;\\
$\left[\alpha_{1} \colon\ldots \colon \alpha_{n+1}  \right]$ & a set of homogeneous coordinates for a point in $\mathbb{P}^{n}(\K)$.		
\end{tabular}

\begin{theorem}\label{THM1} Let $X$ be a set with  at most $\frac{q^{n+1}-q}{q-1}$ elements. If $f_{1}, \ldots, f_{k} \in \mathrm{Map}(X,\F)$ for some $k>n$ then there exist $g_{1}, \ldots, g_{n} \in \mathrm{Span}(f_{1}, \ldots, f_{k})$ such that $\mathrm{Z}(g_{1}, \ldots, g_{n})=\mathrm{Z}(f_{1}, \ldots, f_{k})$.
 \end{theorem}

This theorem is  best possible with respect to  the cardinality of $X$. Indeed, we have the following. 
\begin{proposition}\label{PROP1} For every  field  $\F$ and every positive integer $n$ there are a set $X_{n}$ of cardinality $\frac{q^{n+1}-q}{q-1}+1$, and maps 
	$f_{1}, \ldots, f_{n+1} \in \mathrm{Map}(X_{n},\F)$ such that	
	$\mathrm{Z}(f_{1}, \ldots, f_{n+1})=\emptyset$ but $\mathrm{Z}(g_{1}, \ldots, g_{n})\ne \emptyset$ for any $g_{1}, \ldots, g_{n} \in \mathrm{Span}(f_{1}, \ldots, f_{n+1})$.
\end{proposition}  

We have two immediate corollaries of Theorem \ref{THM1} of interest in  algebraic geometry. 

\begin{corollary}\label{wniosek}
	Let $n > 0$ and let
	$\phi \colon \mathcal{F} \to \mathrm{Map}(\mathbb{A}^{n}(\F),\F)$ be a homomorphism of vector spaces over $\F$. Any subset of $\mathbb{A}^{n}(\F)$ defined by some members of $\mathcal{F}$ (i.e., the zero locus of their images via $\phi$) can be defined using at most $n$ members of $\mathcal{F}$.
\end{corollary} 
The space $\mathcal{F}$ can be, for example, a space of polynomials in $n$ variables of bounded total degree.

\begin{corollary}\label{wniosek1}
	Let $n \ge 0$ and let 
	$\phi \colon \mathcal{F} \to \mathrm{Map}(\mathbb{P}^{n}(\F),\F)$ be a homomorphism of vector spaces over $\F$. Any nonempty subset of $\mathbb{P}^{n}(\F)$ defined by some members of $\mathcal{F}$ (i.e., the zero locus of their images via $\phi$) can be defined using at most $n$ members of $\mathcal{F}$.
\end{corollary}	
The space $\mathcal{F}$ can be a space of homogeneous polynomials in $n+1$ variables of bounded total degree, the space of quadratic (or higher degree) forms in $n+1$ variables, the space of diagonal forms in $n+1$ variables, etc.\\

Before we present the proofs of Theorem \ref{THM1} and Proposition \ref{PROP1}, we separately state their following ingredient.

 Let $\K$ be an arbitrary field, and $n$ be a positive integer. Denote by $\mathcal{M}_{n}$ the set of all matrices in $M_{n, n+1}(\K)$ in  reduced row echelon form having the rank equal to $n$, by $N(M)$ the null space of a matrix $M$, by $\theta$ the zero vector in $\K^{n+1}$, and by $\sim$ the equivalence relation which identifies points
 lying on the same line through the origin.
 
\begin{lemma}\label{lemat}
The map		
\[\begin{array}{c}
	\mathcal{M}_{n} \to \mathbb{P}^{n}(\K)\\
	M \mapsto (N(M)\setminus\left\lbrace \theta\right\rbrace)_{\sim}
\end{array}\]
is bijective.	
\end{lemma}

\begin{proof} 
 Denote by $\mathcal{N}_{n}$ the set of all matrices in $M_{n, n+1}(\K)$ having the rank equal to $n$. For every $M \in\mathcal{N}_{n}$ the dimension of the vector space $N(M) < \K^{n+1}$ equals $1$ by the rank–nullity theorem,  so 
 $(N(M)\setminus\left\lbrace \theta\right\rbrace)_{\sim} \in \mathbb{P}^{n}(\K)$.
 We thus have the map 	
 \[\begin{array}{c}
 \mathcal{N}_{n} \to \mathbb{P}^{n}(\K)\\
 M \mapsto (N(M)\setminus\left\lbrace \theta\right\rbrace)_{\sim}
 \end{array}\]
 Since matrices of the same size have equal null spaces if and only if they are row equivalent, the induced map 
 \[
   \mathcal{N}_{n} / GL_{n}(\K)  \to \mathbb{P}^{n}(\K)
 \]
 is well-defined and injective. It is also surjective, since every vector subspace of $\K^{n+1}$ having dimension equal to $1$ is the null space of a matrix in $\mathcal{N}_{n}$. 

Since the canonical map 
\[
\mathcal{M}_{n}  \to \mathcal{N}_{n}/GL_{n}(\K)
\]
is bijective, the lemma follows.
\end{proof}

\begin{proof}[Proof of Theorem \ref{THM1}] It is enough to prove the statement for $k=n+1$ since we may apply induction.

Denote 
\[S=\left\lbrace 
[f_{1}(x) \colon  \ldots \colon  f_{n+1}(x)] \colon x \in X \setminus \mathrm{Z}(f_{1}, \ldots, f_{n+1}) 
\right\rbrace. \]
By  Lemma \ref{lemat} every element $s$ of $S$ defines a unique matrix in $\mathcal{M}_{n}$; denote this matrix by $M_{s}$.  Examine the set 
\[T=\mathcal{M}_{n}\setminus \left\lbrace M_{s} \colon s \in S \right\rbrace. \] 
By Lemma \ref{lemat} the number of elements in $\mathcal{M}_{n}$ equals the cardinality of $\mathbb{P}^{n}(\F)$, i.e., $\frac{q^{n+1}-1}{q-1}$. The number of elements in $S$ is at most the cardinality of $X$, i.e., $\frac{q^{n+1}-q}{q-1}$.
Hence the cardinality of $T$ is  at least
$\frac{q^{n+1}-1}{q-1}-\frac{q^{n+1}-q}{q-1}=1$. So choose a matrix $M\in T$. Our $g_{1}, \ldots, g_{n}$ are defined by
\[\left[ \begin{array}{c}
g_{1}\\
\vdots\\
g_{n}
\end{array}\right]=M \left[ \begin{array}{c}
f_{1}\\
\vdots\\
f_{n+1}
\end{array}\right].\]  
Indeed, the inclusion $\mathrm{Z}(f_{1}, \ldots, f_{n+1}) \subset \mathrm{Z}(g_{1}, \ldots, g_{n})$ is obvious, and 
by the definition of $T$ the set $\mathrm{Z}(g_{1}, \ldots, g_{n})$ is disjoint from $X \setminus \mathrm{Z}(f_{1}, \ldots, f_{n+1})$, i.e., $\mathrm{Z}(g_{1}, \ldots, g_{n}) \subset \mathrm{Z}(f_{1}, \ldots, f_{n+1})$. 
\end{proof}

In order to prove Proposition \ref{PROP1} we need the following.

\begin{lemma}\label{matrix thy} Let $\K$ be an arbitrary field. 
For any matrix $A \in M_{n, m}(\K)$ where $n\le m$ there exist a matrix $M \in M_{n, m}(\K)$  in  reduced row echelon form having the rank equal to $n$, and a matrix $B \in M_{n, n}(\K)$ such that
$A=BM$.
\end{lemma}

\begin{proof} Denote by $I_{r, k, l}$ the matrix in $M_{k, l}(\K)$ having $x_{11}=\ldots=x_{rr}=1$ and all remaining entries equal to 0. Denote the rank of $A$ by $r$.
	Let $G_{1} \in GL_{n}(\K)$ and $G_{2} \in GL_{m}(\K)$ be matrices transforming $A$ into $I_{r, n, m}$, i.e., $G_{1} A G_{2} = I_{r, n, m}$. Since $I_{r, n, m}=I_{r, n, n}I_{n, n, m}$, we get $A=G_{1}^{-1}I_{r, n, n}I_{n, n, m}G_{2}^{-1}$. Let $G_{3} \in GL_{n}(\K)$ be the matrix transforming $I_{n, n, m}G_{2}^{-1}$ into  reduced row echelon form. We have
	\[A=G_{1}^{-1}I_{r, n, n}G_{3}^{-1} G_{3}^{}I_{n, n, m}G_{2}^{-1}.\]  
	Put $B=G_{1}^{-1}I_{r, n, n}G_{3}^{-1}$ and $M= G_{3}^{}I_{n, n, m}G_{2}^{-1}$.
	\end{proof}	

\begin{proof}[Proof of Proposition \ref{PROP1}]
For every point $P \in \mathbb{P}^{n}(\F)$ choose a set of homogeneous
coordinates for $P$ and denote it by $c_{P}$. Define $X_{n}=\left\lbrace c_{P} \colon P \in \mathbb{P}^{n}(\F) \right\rbrace$. The cardinality of $X_{n}$ is $\frac{q^{n+1}-1}{q-1}=\frac{q^{n+1}-q}{q-1}+1$. Consider $f_{1}, \ldots, f_{n+1} \in \mathrm{Map}(X_{n},\F)$ defined in the following way: for every $x\in X_{n}$ put
\[f_{i}(x) = \text{the } i\text{th coordinate of } x.\]
We have $\mathrm{Z}(f_{1}, \ldots, f_{n+1})=\emptyset$. 

Let $g_{1}, \ldots, g_{n} \in \mathrm{Span}(f_{1}, \ldots, f_{n+1})$, i.e.,  
\[\left[ \begin{array}{c}
g_{1}\\
\vdots\\
g_{n}
\end{array}\right]=A \left[ \begin{array}{c}
f_{1}\\
\vdots\\
f_{n+1}
\end{array}\right]\]
for some matrix $A \in M_{n, n+1}(\F)$. 
By Lemma \ref{matrix thy} there exist a matrix $M \in M_{n, n+1}(\F)$  in  reduced row echelon form having the rank equal to $n$, and a matrix $B \in M_{n, n}(\F)$ such that
$A=BM$.
Hence by  Lemma \ref{lemat} we get that  there is $x \in X_{n}$ belonging to $\mathrm{Z}(g_{1}, \ldots, g_{n})$.    
\end{proof}

\begin{proof}[Proof of Corollary \ref{wniosek}] 	
	For any positive integer $n$ we have $\frac{q^{n+1}-q}{q-1}\ge q^{n}=\left| \mathbb{A}^{n}(\F)\right|$. Applying Theorem \ref{THM1} and some elementary algebra, we get the assertion.
\end{proof}

\begin{remark}
	It has been suggested by the reviewer of this paper to include the following example to demonstrate that although the bound $\frac{q^{n+1}-q}{q-1}\ge q^{n}$ used  in the proof of Corollary \ref{wniosek} is rather crude, the result is sharp for any $q$. Consider the system of $n$ polynomials $f_{i}(x_{1}, \ldots, x_{n})=x_{i}$. While $\mathrm{Z}(f_{1}, \ldots, f_{k})=\left\lbrace \theta\right\rbrace$, any system of $n-1$ combinations of them has at least $q$ common zeros.   
\end{remark}

\begin{proof}[Proof of Corollary \ref{wniosek1}] Let $\left\lbrace f_{1}, \ldots, f_{k}\right\rbrace $ be the image via $\phi$ of a subset of  $\mathcal{F}$. Let $\alpha \in \mathrm{Z}(f_{1}, \ldots, f_{k})$. Denote by $\bar{f_{1}}, \ldots, \bar{f_{k}}$ the images of $f_{1}, \ldots, f_{k}$ via the restriction homomorphism 
	\[\begin{array}{c}
		r \colon \mathrm{Map}(\mathbb{P}^{n}(\F),\F) \to \mathrm{Map}(\mathbb{P}^{n}(\F)\setminus\left\lbrace \alpha\right\rbrace,\F)\\
		r(f)=f |_{\mathbb{P}^{n}(\F)\setminus\left\lbrace \alpha\right\rbrace}.
	\end{array}\]
	For any positive integer $n$ we have \[\left| \mathbb{P}^{n}(\F)\setminus\left\lbrace \alpha\right\rbrace\right|=\left| \mathbb{P}^{n}(\F)\right| -1 = \frac{q^{n+1}-1}{q-1}-1= \frac{q^{n+1}-q}{q-1}.\] 
	So we apply Theorem \ref{THM1} to get 		
	$\bar{g_{1}}, \ldots, \bar{g_{n}}\in \mathrm{Span}(\bar{f_{1}}, \ldots, \bar{f_{k}})$
	such that $\mathrm{Z}(\bar{g_{1}}, \ldots, \bar{g_{n}}) =\mathrm{Z}(\bar{f_{1}}, \ldots, \bar{f_{k}})$. Let $A \in M_{k, n}(\F)$ be such that
	\[\left[ \begin{array}{c}
		\bar{g_{1}}\\
		\vdots\\
		\bar{g_{n}}
	\end{array}\right]=A \left[ \begin{array}{c}
		\bar{f_{1}}\\
		\vdots\\
		\bar{f_{k}}
	\end{array}\right].\] 
	Define $g_{1}, \ldots, g_{n} \in \mathrm{Map}(\mathbb{P}^{n}(\F),\F)$ by
	\[\left[ \begin{array}{c}
		{g_{1}}\\
		\vdots\\
		{g_{n}}
	\end{array}\right]=A \left[ \begin{array}{c}
		{f_{1}}\\
		\vdots\\
		{f_{k}}
	\end{array}\right].\] 
	We are done, since 
	\[\begin{array}{l}
		\mathrm{Z}(f_{1}, \ldots, f_{k})= \left\lbrace \alpha\right\rbrace \cup \mathrm{Z}(\bar{f_{1}}, \ldots, \bar{f_{k}}),  \, \mathrm{and}\\
		\mathrm{Z}(g_{1}, \ldots, g_{n})= \left\lbrace \alpha\right\rbrace \cup \mathrm{Z}(\bar{g_{1}}, \ldots, \bar{g_{n}}).
	\end{array}\]
\end{proof}
 
\section*{Acknowledgements.} We are grateful to Grzegorz Banaszak and Bartosz Naskr\k{e}cki for discussions and suggestions. We wish to thank an anonymous referee for many improvements; in particular, for suggesting the concise formulation and proof of Lemma \ref{lemat}.

\bibliographystyle{plain}

\end{document}